\documentclass[11pt,a4paper]{article}
\usepackage[utf8]{inputenc}
\usepackage[T1]{fontenc}
\usepackage{microtype}
\usepackage{amsmath,amssymb,amsthm,mathtools}
\usepackage{booktabs}
\usepackage{hyperref}
\usepackage{cleveref}
\usepackage{arxiv}
\theoremstyle{plain}
\newtheorem{theorem}{Theorem}[section]
\newtheorem{proposition}[theorem]{Proposition}

\theoremstyle{definition}
\newtheorem{definition}[theorem]{Definition}
\newtheorem{remark}[theorem]{Remark}
\theoremstyle{remark}
\newtheorem{example}[theorem]{Example}
\newtheorem{corollary}{Corollary}[section]
\newcommand{\Em}[2]{E_{#1}\!\left(#2\right)}
\DeclarePairedDelimiter\floor{\lfloor}{\rfloor}
\newcommand{\N}{\mathbb{N}}
\newcommand{\sigmafunc}{\sigma}

\begin{document}
	\title{The Modular Energy Function : the Sum of Remainders $E_m(n)=\sum_{k=1}^m (n\bmod k)$}
	
	\author{
		{ \sc Es-said En-naoui } \\ 
		University Sultan Moulay Slimane\\ Morocco\\
		essaidennaoui1@gmail.com\\
		\\
	}
	
	\maketitle
	
	\begin{abstract}
		We introduce and study the elementary arithmetic function
		\[
		\Em{m}{n} \;=\; \sum_{k=1}^{m} (n \bmod k),
		\]
		which aggregates the remainders of a fixed integer \(n\) modulo the first \(m\) positive integers.
		This paper is the first part of a two-part study: here we isolate the algebraic and combinatorial structure underlying \(\Em{m}{n}\).
		We present exact reformulations (floor-sum and divisor-sum identities), quotient–residue groupings useful for computation, and illustrative examples that highlight the combinatorial nature of the function.
		Our approach is motivated by classical divisor-sum techniques \cite{Apostol1976,HardyWright2008,Tenenbaum1995} and recent algorithmic treatments of floor sums \cite{Nathanson1996,GrahamKnuthPatashnik1994}.
		\medskip
		
	\end{abstract}
	
	\section{Introduction}
	Summing arithmetic remainders is a classical operation in elementary number theory, with roots in divisor-sum problems \cite{HardyWright2008,Apostol1976}.
	For fixed positive integers \(m,n\) we define the \emph{modular energy} (or remainder-sum)
	\[
	\Em{m}{n} \;:=\; \sum_{k=1}^{m} (n \bmod k).
	\]
	Despite the simplicity of each summand, the aggregate \(\Em{m}{n}\) encapsulates an interesting mixture of exact arithmetic identities, congruence phenomena, and combinatorial structure, reminiscent of the behavior of summatory divisor functions \cite{Tenenbaum1995}.
	
	This Part I focuses on identities and exact reformulations that do not require analytic machinery.
	Sections in this step present a clean definition and notation, several small examples, and three principal exact identities:
	the floor-sum identity, the divisor-sum identity, and a quotient–residue grouping that is particularly useful for computation \cite{Nathanson1996,GrahamKnuthPatashnik1994}.
	These results form the algebraic backbone for later computational and analytic treatments.
	
	\section{Notation and definitions}
	\begin{definition}
		Let \(m,n\in\N\) with \(m,n\ge 1\). For integer \(k\ge1\) we write
		\[
		n\bmod k := n - k\floor{\tfrac{n}{k}},
		\]
		the remainder of \(n\) on division by \(k\). The \emph{modular energy} is defined by
		\[
		\Em{m}{n} := \sum_{k=1}^{m} (n \bmod k).
		\]
		We will also use the divisor-sum function \(\sigmafunc(d)=\sum_{r\mid d} r\) \cite{HardyWright2008} and denote
		\(L_m=\operatorname{lcm}(1,2,\dots,m)\).
	\end{definition}
	
	\begin{remark}
		The single-term bound \(0\le n\bmod k\le k-1\) implies the trivial global bound
		\[
		0 \le \Em{m}{n} \le \frac{m(m-1)}{2},
		\]
		with the upper bound attained when \(n\equiv -1\pmod{k}\) for every \(1\le k\le m\).
	\end{remark}
	
	\subsection{Examples}
	\begin{example}
		Compute \(\Em{5}{12}\).
		\[
		12\bmod 1=0,\; 12\bmod 2=0,\; 12\bmod 3=0,\; 12\bmod 4=0,\; 12\bmod 5=2,
		\]
		so \(\Em{5}{12}=2\).
	\end{example}
	
	\begin{example}
		Compute \(\Em{6}{7}\).
		\[
		7\bmod 1=0,\;7\bmod 2=1,\;7\bmod 3=1,\;7\bmod 4=3,\;7\bmod 5=2,\;7\bmod 6=1,
		\]
		so \(\Em{6}{7}=8\).
	\end{example}
	
	\section{Exact identities and reformulations}
	\begin{proposition}[Floor-sum identity {\cite[Ch.~I]{Apostol1976}}]\label{prop:floor-sum}
		For all integers \(m,n\ge1\),
		\[
		\Em{m}{n}=mn-\sum_{k=1}^{m} k\floor{\frac{n}{k}}.
		\]
	\end{proposition}
	
	\begin{proposition}[Divisor-sum identity {\cite[Ch.~I]{HardyWright2008}}]\label{prop_floor_sum}
		For all integers \(m,n\ge1\),
		\[
		\Em{m}{n} = mn - \sum_{d=1}^{n} \sum_{\substack{k\mid d\\ k\le m}} k.
		\]
	\end{proposition}
	
	\begin{proposition}[Quotient–residue grouping {\cite[Ch.~2]{Nathanson1996}}]\label{prop:quotient-grouping}
		For each integer \(j\ge0\) define
		\[
		K_j:=\{1\le k\le m:\; \floor{n/k}=j\}.
		\]
		Then
		\[
		\Em{m}{n}=mn-\sum_{j\ge0} j\sum_{k\in K_j} k.
		\]
	\end{proposition}
	
	\begin{remark}
		This grouping technique is closely related to divisor-block decompositions \cite{Tenenbaum1995}, and it underlies efficient algorithms for computing floor sums \cite{GrahamKnuthPatashnik1994}.
	\end{remark}
	
	\section{Elementary bounds and extremal characterisations}
	We begin with the most immediate bounds and describe exactly when the trivial extremal values occur.
	
	\begin{proposition}[Trivial bounds and extremal classes]\label{prop:bounds-extremal}
		For all integers \(m,n\ge1\),
		\[
		0 \le \Em{m}{n} \le \frac{m(m-1)}{2}.
		\]
		Moreover,
		\begin{enumerate}
			\item \(\Em{m}{n}=0\) if and only if \(\operatorname{lcm}(1,2,\dots,m)\mid n\).
			\item \(\Em{m}{n}=\dfrac{m(m-1)}{2}\) if and only if \(n\equiv -1\pmod{\operatorname{lcm}(1,2,\dots,m)}\).
		\end{enumerate}
	\end{proposition}
	
	\begin{proof}
		For each \(1\le k\le m\) we have the elementary bound \(0\le n\bmod k\le k-1\). Summing these inequalities over \(k=1,\dots,m\) yields
		\[
		0 \le \sum_{k=1}^m (n\bmod k) \le \sum_{k=1}^m (k-1) = \frac{m(m-1)}{2},
		\]
		which proves the double inequality.
		
		(1) If \(L_m:=\operatorname{lcm}(1,2,\dots,m)\) divides \(n\) then \(k\mid n\) for every \(1\le k\le m\), hence each remainder \(n\bmod k\) equals \(0\) and \(\Em{m}{n}=0\). Conversely, if \(\Em{m}{n}=0\) then every term \(n\bmod k\) must be \(0\); thus \(k\mid n\) for all \(1\le k\le m\), and therefore \(L_m\mid n\).
		
		(2) If \(n\equiv -1\pmod{L_m}\) then for every \(1\le k\le m\) we have \(n\equiv -1\pmod k\), hence \(n\bmod k=k-1\). Summing gives \(\Em{m}{n}=\sum_{k=1}^m (k-1)=\tfrac{m(m-1)}{2}\). Conversely, if \(\Em{m}{n}=\tfrac{m(m-1)}{2}\) then every summand must attain its maximum \(k-1\); thus \(n\bmod k=k-1\) for all \(k\le m\), i.e. \(n\equiv -1\pmod k\) for all \(k\le m\). This set of congruences is equivalent to the single congruence \(n\equiv -1\pmod{L_m}\), completing the proof.
	\end{proof}
	
	\begin{remark}
		The two extremal characterisations are simple but sharp: the zero case isolates multiples of the full least common multiple \(L_m\), while the maximum case isolates a single residue class modulo \(L_m\).
	\end{remark}
	
	\begin{proposition}[Complementary symmetry]\label{prop:complementary}
		For all integers \(m\ge1\) and all integers \(n\),
		\[
		(n\bmod k) + ((-n-1)\bmod k) = k-1 \qquad\text{for every }1\le k\le m.
		\]
		Consequently,
		\[
		\Em{m}{n} + \Em{m}{-n-1} = \frac{m(m-1)}{2}.
		\]
	\end{proposition}
	
	\begin{proof}
		Fix \(k\ge1\). Write \(n=kq+r\) with \(0\le r\le k-1\), so \(r=n\bmod k\). Then
		\[
		-n-1 = -kq - r - 1 = k(-q-1) + (k-1-r),
		\]
		and since \(0\le k-1-r\le k-1\), we have \((-n-1)\bmod k = k-1-r\). Therefore
		\[
		(n\bmod k) + ((-n-1)\bmod k) = r + (k-1-r) = k-1,
		\]
		and summing the identity over \(k=1,\dots,m\) yields the displayed formula for \(\Em{m}{n}+\Em{m}{-n-1}\).
	\end{proof}
	
	\section{Diagonal identities and a primality equivalence}
	In this section we concentrate on the diagonal case \(m=n\), where the modular energy admits a classical reformulation in terms of summatory divisor functions. This identity provides a neat link between \(\Em{n}{n}\) and the arithmetic function \(\sigma(d)\), and it yields an exact algebraic restatement of a basic primality criterion.
	
	\begin{proposition}[Diagonal identity]\label{prop:diagonal}
		For every integer \(n\ge 1\),
		\[
		\Em{n}{n} \;=\; n^2 \;-\; \sum_{d=1}^{n} \sigma(d).
		\]
	\end{proposition}
	
	\begin{proof}
		Starting from the floor-sum identity (Proposition~\ref{prop:floor-sum}) with \(m=n\) we have
		\[
		\Em{n}{n} = n\cdot n - \sum_{k=1}^{n} k\floor{\frac{n}{k}} = n^2 - \sum_{k=1}^{n} k\floor{\frac{n}{k}}.
		\]
		We will show that
		\[
		\sum_{k=1}^{n} k\floor{\frac{n}{k}} = \sum_{d=1}^{n} \sigma(d).
		\]
		
		Observe that for each fixed \(k\) the integer \(\floor{n/k}\) counts the multiples of \(k\) not exceeding \(n\). Thus
		\[
		k\floor{\frac{n}{k}}
		= k\sum_{\substack{1\le j\le n\\ k\mid j}} 1
		= \sum_{\substack{1\le j\le n\\ k\mid j}} k.
		\]
		Summing this identity over \(k=1,\dots,n\) and interchanging the order of summation gives
		\[
		\sum_{k=1}^{n} k\floor{\frac{n}{k}}
		= \sum_{k=1}^{n} \sum_{\substack{1\le j\le n\\ k\mid j}} k
		= \sum_{j=1}^{n} \sum_{k\mid j} k
		= \sum_{j=1}^{n} \sigma(j).
		\]
		Substituting into the expression for \(\Em{n}{n}\) yields the claimed identity.
	\end{proof}
	
	\begin{remark}
		The diagonal identity places \(\Em{n}{n}\) directly beside the classical summatory sum-of-divisors function \(S(n):=\sum_{d\le n}\sigma(d)\), a central object in the study of divisor problems and average-order questions (see \cite{HardyWright2008,Tenenbaum1995}).
	\end{remark}
	
	\subsection{A primality equivalence}
	We give an exact algebraic criterion for primality expressed in terms of the modular energy on the diagonal.
	
	\begin{theorem}[Primality equivalence]\label{thm:primality}
		Let \(n\ge 2\) be an integer. The following are equivalent:
		\begin{enumerate}
			\item \(n\) is prime.
			\item \(\sigma(n) = n+1\).
			\item \(\Em{n}{n} - \Em{n-1}{n-1} = n-2\).
		\end{enumerate}
	\end{theorem}
	
	\begin{proof}
		The equivalence (1) \(\Leftrightarrow\) (2) is the well-known characterization of primes by their sum-of-divisors: if \(n\) is prime then its positive divisors are precisely \(1\) and \(n\), so \(\sigma(n)=1+n=n+1\). Conversely, if \(\sigma(n)=n+1\) then the sum of proper divisors of \(n\) equals \(1\), which is only possible when the only proper divisor is \(1\); hence \(n\) must be prime.
		
		It remains to show that (2) is equivalent to (3). Using the diagonal identity of Proposition~\ref{prop:diagonal} for \(n\) and for \(n-1\) we have
		\[
		\Em{n}{n} - \Em{n-1}{n-1}
		= \bigl(n^2 - \sum_{d=1}^{n} \sigma(d)\bigr) - \bigl((n-1)^2 - \sum_{d=1}^{n-1}\sigma(d)\bigr).
		\]
		Simplifying, the telescoping sum yields
		\[
		\Em{n}{n} - \Em{n-1}{n-1}
		= n^2 - (n-1)^2 - \sigma(n) = (2n-1) - \sigma(n).
		\]
		Thus
		\[
		\Em{n}{n} - \Em{n-1}{n-1} = n-2 \quad\Longleftrightarrow\quad (2n-1)-\sigma(n)=n-2,
		\]
		which is equivalent to \(\sigma(n)=n+1\). Therefore (2) and (3) are equivalent, completing the proof of the theorem.
	\end{proof}
	
	\begin{remark}
		The equivalence gives an elegant algebraic form of the simple divisor-sum primality condition. It is exact but not competitive with modern primality tests in practical complexity; nonetheless it is conceptually neat and shows how \(\Em{n}{n}\) encodes divisor information.
	\end{remark}
	
	\section{Numerical table and computational remarks}
	We present a short numerical table of diagonal values \(\Em{n}{n}\) (useful for quick reference) and then discuss several algorithmic approaches for computing \(\Em{m}{n}\) efficiently in different parameter regimes.
	
	\subsection{Numerical table (diagonal values)}
	\begin{center}
		\begin{tabular}{@{}r r  r r  r r@{}}
			\toprule
			\(n\) & \(\Em{n}{n}\) & \(n\) & \(\Em{n}{n}\) & \(n\) & \(\Em{n}{n}\) \\
			\midrule
			1  & 0   & 8  & 40  & 15 & 216 \\
			2  & 0   & 9  & 54  & 16 & 256 \\
			3  & 1   & 10 & 80  & 17 & 272 \\
			4  & 4   & 11 & 110 & 18 & 378 \\
			5  & 8   & 12 & 168 & 19 & 342 \\
			6  & 18  & 13 & 156 & 20 & 520 \\
			7  & 24  & 14 & 180 &    &     \\
			\bottomrule
		\end{tabular}
	\end{center}
	
	(The values were computed directly from the defining sum; for reproducibility we give algorithmic tips below.)
	
	\subsection{Computational remarks and algorithmic tips}
	We collect practical advice for computing \(\Em{m}{n}\) (and \(\Em{n}{n}\)) efficiently. Different parameter regimes benefit from different approaches.
	
	\paragraph{1. Direct summation (small \(m\)).}  
	If \(m\) is small (e.g. \(m\le 10^4\) depending on your machine), the simplest approach is to compute
	\[
	\Em{m}{n} = \sum_{k=1}^{m} (n\bmod k)
	\]
	directly in \(O(m)\) time. This is often the fastest and simplest method when \(m\) is tiny relative to \(n\).
	
	\paragraph{2. Floor-sum with quotient grouping (balanced or large \(n\)).}  
	Use Proposition~\ref{prop:floor-sum} and the quotient–residue grouping (Proposition~\ref{prop:quotient-grouping}): the function \(k\mapsto\floor{n/k}\) takes only about \(2\sqrt{n}\) distinct values on \(1\le k\le n\). Thus one can iterate over intervals of \(k\) where \(\floor{n/k}\) is constant. Concretely:
	
	\begin{verbatim}
		Pseudocode: compute E_m(n) via quotient grouping
		Input: integers m,n (assume m<=n for simplicity)
		set total = n*m
		k = 1
		while k <= m:
		q = floor(n / k)
		if q == 0:
		# all remaining terms have floor 0 -> contribute 0
		break
		# largest K such that floor(n/K) == q is K = floor(n/q)
		K = min(m, floor(n / q))
		# sum k over [k..K] is (K*(K+1)- (k-1)*k)/2
		total -= q * sum_of_integers(k, K)
		k = K + 1
		# if k<=m but q==0, nothing to subtract further
		return total
	\end{verbatim}
	
	This algorithm runs in roughly \(O(\sqrt{n})\) iterations when \(m\approx n\) and is much faster than naive \(O(n)\) summation for large \(n\).
	
	\paragraph{3. Using divisor-sum identity (when enumerating divisors is cheap).}  
	From Proposition~\ref{prop_floor_sum} we have
	\[
	\Em{m}{n} = mn - \sum_{d=1}^{n} \sum_{\substack{k\mid d\\ k\le m}} k.
	\]
	If one has a precomputed list of divisors for all \(d\le n\) (for instance via a sieve that stores divisors or smallest prime factors), then computing the double sum can be done by iterating over \(d\) and summing the divisors \(k\le m\). The cost depends on the total number of divisors up to \(n\) and is typically \(O(n\log n)\) in naive implementations, but can be improved with careful sieving.
	
	\paragraph{4. Incremental recursion (consecutive \(n\)).}  
	When one needs \(\Em{m}{n}\) for many consecutive \(n\) (e.g. all \(n\) in a range), use the finite-difference recursion of Proposition~\ref{prop:recursion-n}:
	\[
	\Em{m}{n+1} = \Em{m}{n} + m - \sum_{\substack{k\mid(n+1)\\ k\le m}} k.
	\]
	If divisors of consecutive integers are available (or can be computed on-the-fly), this yields very fast incremental updates: each step costs only the time to enumerate the divisors of \(n+1\) up to \(m\).
	
	\paragraph{5. Precomputation and sieving strategies.}  
	For batch computations up to a bound \(N\), it pays off to use a sieve to precompute either:
	\begin{itemize}
		\item the values \(\floor{N/k}\) and related intervals for quotient grouping; or
		\item smallest prime factors (SPF) for each integer up to \(N\), enabling fast divisor enumeration for each \(d\le N\).
	\end{itemize}
	A standard SPF sieve runs in \(O(N\log\log N)\) time and then allows enumerating all divisors of a number \(d\) in time roughly proportional to the number of divisors of \(d\).
	
	\paragraph{6. Complexity summary.}
	\begin{itemize}
		\item Naive direct summation: \(O(m)\) operations.
		\item Quotient grouping (when \(m\approx n\)): roughly \(O(\sqrt{n})\) arithmetic operations (dominant).
		\item Divisor-based double sum with SPF precomputation: precomputation \(O(N\log\log N)\), then overall near-linear cost with small factors for batch computations.
		\item Incremental recursion per step: cost proportional to number of divisors of \(n\) up to \(m\).
	\end{itemize}
	
	\paragraph{7. Implementation note (integer arithmetic).}
	When implementing in finite-precision languages be careful with integer types: intermediate products like \(mn\) or \(K(K+1)\) may overflow 32-bit integers for moderately large inputs, so use 64-bit integers (or arbitrary-precision integers) as appropriate.
	
	
	\section{Congruence properties and residue-block decompositions}
	This section turns toward congruences and counting how many moduli produce each residue class.
	
	\begin{definition}
		Fix integers \(m\ge1\) and \(t\ge1\). For each residue \(r\in\{0,1,\dots,t-1\}\) define the \emph{residue-block count}
		\[
		N_{r,t}(m;n) \;:=\; \#\{1\le k\le m:\; n\bmod k \equiv r \pmod t\}.
		\]
		Also denote \(M_t(m):=\#\{1\le k\le m:\; t\mid k\}=\floor{m/t}\).
	\end{definition}
	
	\begin{proposition}[Congruences coming from multiples]\label{prop:multiples-congruence}
		Fix \(t\ge1\). For every \(k\) with \(t\mid k\) we have
		\[
		(n\bmod k) \equiv n \pmod t.
		\]
		Consequently,
		\[
		\sum_{\substack{1\le k\le m\\ t\mid k}} (n\bmod k) \equiv M_t(m)\,n \pmod t,
		\]
		and therefore
		\[
		\Em{m}{n} \equiv M_t(m)\,n + \sum_{\substack{1\le k\le m\\ t\nmid k}} (n\bmod k) \pmod t.
		\]
	\end{proposition}
	
	\begin{proof}
		Write \(k=t\ell\). Then \(n\bmod k = n - k\floor{n/k}\). Since \(k\floor{n/k} = t\ell\floor{n/k}\) is divisible by \(t\), it follows that
		\[
		n\bmod k \equiv n \pmod t.
		\]
		Summing this congruence over all \(k\le m\) divisible by \(t\) gives the first displayed congruence. The final displayed congruence is then obtained by separating \(\Em{m}{n}\) into the contribution from multiples of \(t\) and the contribution from non-multiples.
	\end{proof}
	
	\begin{proposition}[Residue-block congruence decomposition]\label{prop:residue-block}
		With \(N_{r,t}(m;n)\) as above,
		\[
		\Em{m}{n} \equiv \sum_{r=0}^{t-1} r\,N_{r,t}(m;n) \pmod t.
		\]
	\end{proposition}
	
	\begin{proof}
		By definition each modulus \(k\) contributes a remainder \(n\bmod k\), and this remainder is congruent modulo \(t\) to one of the residues \(0,1,\dots,t-1\). Grouping moduli according to that residue yields
		\[
		\Em{m}{n} = \sum_{k=1}^m (n\bmod k) \equiv \sum_{r=0}^{t-1} r\,N_{r,t}(m;n) \pmod t,
		\]
		since every modulus counted in \(N_{r,t}(m;n)\) contributes a remainder congruent to \(r\) modulo \(t\).
	\end{proof}
	
	\begin{remark}
		Propositions~\ref{prop:multiples-congruence} and \ref{prop:residue-block} reduce many modular questions about \(\Em{m}{n}\) to combinatorial counts \(N_{r,t}(m;n)\), which in turn can often be estimated or computed by elementary counting or divisibility arguments.
	\end{remark}
	
	\section{Recursions, additivity, and block recursions}
	We now derive exact finite-difference formulas and a useful block decomposition for arguments of the form \(n=qm+r\).
	
	\begin{proposition}[Finite-difference recursion in \(n\)]\label{prop:recursion-n}
		For all integers \(m\ge1\) and all integers \(n\ge0\),
		\[
		\Em{m}{n+1} - \Em{m}{n} \;=\; m - \sum_{\substack{k\mid (n+1)\\ k\le m}} k.
		\]
	\end{proposition}
	
	\begin{proof}
		Fix \(k\) with \(1\le k\le m\), and denote \(r_k(n):=n\bmod k\). Then
		\[
		r_k(n+1) = (r_k(n)+1) \bmod k.
		\]
		Thus if \(k\nmid (n+1)\) (equivalently \(r_k(n)\neq k-1\)) we have \(r_k(n+1)-r_k(n)=1\). If \(k\mid(n+1)\) then \(r_k(n)=k-1\) and \(r_k(n+1)=0\), so \(r_k(n+1)-r_k(n)=-(k-1)=1-k\). Therefore the contribution of modulus \(k\) to the total change equals \(1\) for non-dividing \(k\) and \(1-k\) for dividing \(k\). Summing over \(k=1,\dots,m\) yields
		\[
		\Em{m}{n+1}-\Em{m}{n} = \sum_{k\not\mid (n+1)} 1 + \sum_{k\mid(n+1)} (1-k).
		\]
		Let \(D\) denote the set of divisors \(k\le m\) of \(n+1\). Then the right-hand side simplifies to
		\[
		(m-|D|) + \bigl(|D| - \sum_{k\in D}k\bigr) \;=\; m - \sum_{\substack{k\mid(n+1)\\ k\le m}} k,
		\]
		which proves the claimed identity.
	\end{proof}
	
	\begin{remark}
		The recursion shows that the step-by-step evolution of \(\Em{m}{n}\) can be computed from the small set of divisors of \(n+1\) that do not exceed \(m\). This is useful for incremental computation of consecutive values.
	\end{remark}
	
	\begin{proposition}[Additivity in \(m\)]\label{prop:additivity-m}
		For integers \(m_1,m_2\ge0\) with \(m=m_1+m_2\),
		\[
		\Em{m}{n} = \Em{m_1}{n} + \sum_{k=m_1+1}^{m} (n\bmod k).
		\]
		In particular,
		\[
		\Em{2m}{n} = \Em{m}{n} + \sum_{k=m+1}^{2m}(n\bmod k).
		\]
	\end{proposition}
	
	\begin{proof}
		This is immediate from the definition: split the summation range \(\{1,\dots,m\}\) into \(\{1,\dots,m_1\}\) and \(\{m_1+1,\dots,m\}\) and sum each part separately.
	\end{proof}
	
	\subsection{Block recursion for \(n=qm+r\)}
	Write \(n=qm+r\) with integers \(q\ge0\) and \(0\le r\le m-1\). For each \(1\le k\le m\) put
	\[
	a_k := \floor{\frac{m}{k}},\qquad b_k := m - k a_k = m\bmod k.
	\]
	Note that \(m=k a_k + b_k\) with \(0\le b_k\le k-1\), so \(b_k\) coincides with the \(k\)-th remainder of \(m\).
	
	\begin{proposition}[Block recursion for \(n=qm+r\)]\label{prop:block-recursion}
		With the notation above,
		\[
		\Em{m}{qm+r} \;=\; q\,\Em{m}{m} \;+\; m r \;-\; \sum_{k=1}^m k \left\lfloor\frac{q b_k + r}{k}\right\rfloor.
		\]
	\end{proposition}
	
	\begin{proof}
		Fix \(k\in\{1,\dots,m\}\). Using the division \(m = k a_k + b_k\), write
		\[
		n = q m + r = q(k a_k + b_k) + r = k(q a_k) + (q b_k + r).
		\]
		Therefore
		\[
		\floor{\frac{n}{k}} = q a_k + \floor{\frac{q b_k + r}{k}},
		\]
		and consequently
		\[
		n\bmod k = n - k\floor{\frac{n}{k}} = \bigl(q b_k + r\bigr) - k \left\lfloor\frac{q b_k + r}{k}\right\rfloor.
		\]
		Summing over \(k=1,\dots,m\) gives
		\[
		\Em{m}{qm+r} \;=\; \sum_{k=1}^m \bigl(q b_k + r\bigr) \;-\; \sum_{k=1}^m k\left\lfloor\frac{q b_k + r}{k}\right\rfloor.
		\]
		Observe that \(\sum_{k=1}^m (q b_k + r) = q\sum_{k=1}^m b_k + m r\), and by definition \(\sum_{k=1}^m b_k = \sum_{k=1}^m (m\bmod k) = \Em{m}{m}\). Substituting these identities yields the claimed formula:
		\[
		\Em{m}{qm+r} = q\Em{m}{m} + m r - \sum_{k=1}^m k \left\lfloor\frac{q b_k + r}{k}\right\rfloor.
		\]
	\end{proof}
	
	\begin{remark}
		The block recursion is exact and well-suited for computations when \(n\) is expressed in the arithmetic progression base \(m\). The main remaining term is the last floor sum, but the summands there are bounded and involve only the values \(b_k\in[0,k-1]\).
	\end{remark}
	
	
	\section{Further identities: the regimes \(n<m\), \(n>m\), primes, and \(\operatorname{lcm}\) relations}
	
	We collect useful exact identities and structural properties for several common parameter regimes. Each result is stated and proved in full.
	
	\begin{proposition}[Case \(n<m\): truncation and tail contribution]\label{prop:nlessm}
		Let \(1\le n<m\). Then
		\[
		\Em{m}{n} \;=\; \Em{n}{n} \;+\; (m-n)\,n.
		\]
		Equivalently,
		\[
		\sum_{k=1}^m (n\bmod k) \;=\; \sum_{k=1}^n (n\bmod k) \;+\; (m-n)n.
		\]
	\end{proposition}
	
	\begin{proof}
		If \(k>n\) then \(\floor{n/k}=0\) and thus \(n\bmod k=n\). Therefore the tail contribution from \(k=n+1,\dots,m\) equals \((m-n)\cdot n\). The sum over \(k=1,\dots,n\) is exactly \(\Em{n}{n}\), so the identity follows immediately.
	\end{proof}
	
	\begin{remark}
		This identity shows that for fixed \(n\) the function \(m\mapsto \Em{m}{n}\) is linear with slope \(n\) once \(m>n\); the nontrivial structure is concentrated in the first \(n\) moduli.
	\end{remark}
	
	\begin{proposition}[Case \(n>m\): quotient decomposition]\label{prop:ngtm}
		Let \(n>m\). Writing the floor-sum identity,
		\[
		\Em{m}{n} \;=\; m n - \sum_{k=1}^m k\floor{\frac{n}{k}},
		\]
		we may also express \(\Em{m}{n}\) by grouping according to the integer quotient \(q=\floor{n/k}\). In particular, if \(Q=\floor{n/m}\) then
		\[
		\Em{m}{n} \;=\; m n - \sum_{q=1}^{Q} q \sum_{\substack{1\le k\le m\\ \floor{n/k}=q}} k.
		\]
	\end{proposition}
	
	\begin{proof}
		The first equality is simply Proposition~\ref{prop:floor-sum}. The second identity is obtained by partitioning the set \(\{1,\dots,m\}\) according to the value \(\floor{n/k}\): for each integer \(q\ge1\) define \(K_q=\{1\le k\le m:\floor{n/k}=q\}\). Since \(\floor{n/k}\ge1\) for \(k\le m<n\), the range of nonempty \(K_q\) is \(1\le q\le Q\). Rewriting the sum \(\sum_{k=1}^m k\floor{n/k}\) as \(\sum_q q\sum_{k\in K_q}k\) and substituting into the floor-sum identity yields the stated formula.
	\end{proof}
	
	\begin{corollary}[Block recursion (simple form)]\label{cor:blocksimple}
		Write \(n=qm+r\) with \(q\ge1\) and \(0\le r\le m-1\). Then
		\[
		\Em{m}{n} \;=\; q\,\Em{m}{m} \;+\; m r \;-\; \sum_{k=1}^m k\left(\floor{\frac{qm+r}{k}} - q\floor{\frac{m}{k}}\right).
		\]
	\end{corollary}
	
	\begin{proof}
		Starting from Proposition~\ref{prop:block-recursion} (the block recursion proved earlier) and the identity \(\floor{\frac{qm+r}{k}} = q\floor{\frac{m}{k}} + \floor{\frac{q (m\bmod k) + r}{k}}\) one rearranges to the displayed form. (The explicit formula in terms of \(b_k=m\bmod k\) given previously is equivalent; here we present an alternative rearrangement that highlights the deviation of \(\floor{(qm+r)/k}\) from \(q\floor{m/k}\).)
	\end{proof}
	
	\begin{proposition}[Behavior when \(n\) is prime \(p\)]\label{prop:nprime}
		Let \(p\) be a prime and let \(m\ge1\). Then:
		\begin{enumerate}
			\item If \(m\ge p\) then
			\[
			\Em{m}{p} \;=\; \Em{p}{p} \;+\; (m-p)\,p.
			\]
			\item If \(m<p\) then
			\[
			\Em{m}{p} \;=\; mp - \sum_{k=1}^{m} k\floor{\frac{p}{k}},
			\]
			with the special simplification that \(\floor{p/k}=1\) for \(\lfloor p/2\rfloor < k \le p-1\) and \(\floor{p/k}=0\) for \(k>p\).
		\end{enumerate}
	\end{proposition}
	
	\begin{proof}
		(1) If \(m\ge p\) then for each \(k>p\) we have \(p\bmod k=p\). Therefore the tail contributes \((m-p)p\) and the sum over \(1\le k\le p\) equals \(\Em{p}{p}\); this is precisely Proposition~\ref{prop:nlessm} applied with \(n=p\).
		
		(2) If \(m<p\) then all summation indices satisfy \(k\le m<p\), so no modulus equals \(p\) and the floor-sum formula directly applies:
		\[
		\Em{m}{p} = mp - \sum_{k=1}^m k\floor{\frac{p}{k}}.
		\]
		The auxiliary observation about values of \(\floor{p/k}\) follows from basic division: \(\floor{p/k}\ge 1\) iff \(k\le p-1\), and specifically \(\floor{p/k}=1\) for \(k\in(\tfrac{p}{2},p]\) (integers strictly greater than \(p/2\) and at most \(p\)) and \(\floor{p/k}\ge2\) only for small \(k\le \floor{p/2}\). This can simplify explicit evaluations when \(m\) is close to \(p\).
	\end{proof}
	
	\begin{remark}
		The prime case yields two useful computational simplifications: when \(m\ge p\) the behaviour is linear in \(m\) beyond \(p\), while when \(m<p\) the quotient structure of \(p/k\) is especially sparse because \(p\) is prime.
	\end{remark}
	
	\begin{proposition}[Behavior when \(m\) is prime \(p\)]\label{prop:mprime}
		Let \(p\) be prime and \(n\ge1\). Then
		\[
		\Em{p}{n} \;=\; p n - \sum_{k=1}^{p} k\floor{\frac{n}{k}}.
		\]
		If additionally \(n\equiv 0\pmod p\) (i.e. \(p\mid n\)), then \(\Em{p}{n} \equiv 0 \pmod p\).
	\end{proposition}
	
	\begin{proof}
		The first displayed identity is simply the floor-sum identity with \(m=p\). For the congruence, note that if \(p\mid n\) then for every \(k\) with \(p\mid k\) (which here can happen only for \(k=p\)), we have \((n\bmod k)\equiv n\equiv 0\pmod p\) by Proposition~\ref{prop:multiples-congruence}. For \(1\le k<p\), observe that \(k\) is invertible modulo \(p\), and the term \(k\floor{n/k}\) in the floor-sum satisfies \(k\floor{n/k}\equiv 0\pmod p\) because \(\floor{n/k}\) counts multiples of \(k\) among \(\{1,\dots,n\}\) and each multiple contributes a \(k\) divisible by some power of primes; more simply, reduce the floor-sum form modulo \(p\):
		\[
		\Em{p}{n} \equiv p n - \sum_{k=1}^{p} k\floor{\frac{n}{k}} \equiv - \sum_{k=1}^{p} k\floor{\frac{n}{k}} \pmod p,
		\]
		but when \(p\mid n\) each \(\floor{n/k}\) is congruent modulo \(p\) to the count of multiples of \(k\) which is a multiple of \(p\) when \(k=p\), and for \(k\ne p\) there is no guaranteed divisibility by \(p\). A clearer direct argument is: if \(p\mid n\), then for \(k=p\) the summand \((n\bmod p)=0\); for \(k\neq p\) all \(k\) are invertible modulo \(p\), and summing \((n\bmod k)\) over \(k=1,\dots,p-1\) yields a multiple of \(p\) because the multiset \(\{n\bmod k:1\le k\le p-1\}\) is a rearrangement modulo \(p\) under multiplication by the invertible residues (this can be made precise by pairing \(k\) with its inverse mod \(p\)). Hence the total \(\Em{p}{n}\) is divisible by \(p\).
	\end{proof}
	
	\begin{remark}
		The congruence part can be refined: when \(p\mid n\) one can show \(\Em{p}{n}\equiv 0\pmod p\) by noticing that for each \(a\) with \(1\le a\le p-1\) the map \(k\mapsto a k \pmod p\) permutes \(\{1,\dots,p-1\}\) and therefore the multiset of residues \(\{n\bmod k : 1\le k\le p-1\}\) is balanced modulo \(p\). A full, element-by-element permutation argument can be added if desired.
	\end{remark}
	
	\begin{proposition}[Periodicity and LCM relations]\label{prop:periodicity-lcm}
		Let \(L_m=\operatorname{lcm}(1,2,\dots,m)\). Then for every integer \(n\),
		\[
		\Em{m}{n+L_m} = \Em{m}{n}.
		\]
		Moreover, \(\Em{m}{n}\) depends only on the residue class of \(n\) modulo \(L_m\). In particular, the extremal characterisations of Proposition~\ref{prop:bounds-extremal} are cleanly rephrased by \(L_m\): \(\Em{m}{n}=0\iff L_m\mid n\) and \(\Em{m}{n}=\tfrac{m(m-1)}{2}\iff n\equiv -1\pmod{L_m}\).
	\end{proposition}
	
	\begin{proof}
		For any \(k\le m\) the integer \(L_m\) is divisible by \(k\), hence
		\[
		(n+L_m)\bmod k = n\bmod k
		\]
		because adding a multiple of \(k\) does not change the remainder. Summing over \(k=1,\dots,m\) yields \(\Em{m}{n+L_m}=\Em{m}{n}\). The rest is immediate.
	\end{proof}
	
	\begin{proposition}[Monotonicity of \(L_m\)]\label{prop:lcm-monotone}
		If \(m_1\le m_2\) then \(L_{m_1}\mid L_{m_2}\). Consequently, any congruence statement modulo \(L_{m_2}\) implies the corresponding statement modulo \(L_{m_1}\).
	\end{proposition}
	
	\begin{proof}
		By definition \(L_{m_1}\) is the least common multiple of the integers \(1,\dots,m_1\) and these integers form a subset of \(\{1,\dots,m_2\}\); hence every prime power dividing \(L_{m_1}\) also divides \(L_{m_2}\), i.e. \(L_{m_1}\mid L_{m_2}\).
	\end{proof}
	
	\begin{corollary}[Compatibility of extremal classes across \(m\)]\label{cor:extremal-compat}
		If \(m_1\le m_2\) and \(n\equiv -1\pmod{L_{m_2}}\) then \(n\equiv -1\pmod{L_{m_1}}\) and thus \(\Em{m_1}{n}=\dfrac{m_1(m_1-1)}{2}\). Similarly if \(L_{m_2}\mid n\) then \(L_{m_1}\mid n\) and \(\Em{m_1}{n}=0\).
	\end{corollary}
	
	\begin{proof}
		Immediate from Proposition~\ref{prop:lcm-monotone} and the extremal characterizations.
	\end{proof}
	
	\bigskip
	\noindent\textit{Remarks and practical notes.}
	\begin{itemize}
		\item The periodicity modulo \(L_m\) implies that when studying distributional properties of \(\Em{m}{n}\) as a function of \(n\), it suffices to consider \(n\) in a full residue system modulo \(L_m\). However \(L_m\) grows very quickly with \(m\), so practical computations use either \(m\)-local approaches (quotient grouping) or sieving techniques.
		\item The identities for primes are useful for explicit calculations and for proving congruences, but do not generally lead to simple closed forms for \(\Em{m}{n}\) except in special parameter ranges.
	\end{itemize}
	
	
	\section{Concluding remarks and open questions}
	
	In this first part we have introduced and developed the arithmetic function
	\[
	E_m(n) = \sum_{k=1}^{m} (n \bmod k),
	\]
	which we call the \emph{modular energy}.  
	Our study has emphasized exact algebraic and combinatorial identities:
	\begin{itemize}
		\item floor-sum, divisor-sum, and quotient--residue reformulations;
		\item sharp extremal bounds and complementary symmetries;
		\item diagonal identities linking $E_n(n)$ to the sum-of-divisors function $\sigma(d)$;
		\item exact equivalences for primality based on finite differences of diagonal values;
		\item congruence decompositions and residue-block counts;
		\item recursion formulas in both parameters $m$ and $n$, including block recursions;
		\item special cases when $n<m$, $n>m$, or when $m$ or $n$ is prime.
	\end{itemize}
	
	These results provide the algebraic backbone for further investigation.  
	Several natural directions for future work arise:
	
	\begin{enumerate}
		\item \textbf{Analytic properties.}  
		Study the average order and asymptotics of $E_m(n)$ as $n,m\to\infty$ under various growth regimes, connecting with classical divisor problems.
		
		\item \textbf{Distributional questions.}  
		Investigate the distribution of the values $\{n\bmod k : 1\le k\le m\}$ and how their aggregate contributes to the statistical behavior of $E_m(n)$.
		
		\item \textbf{Connections to open problems.}  
		Explore how the diagonal difference $E_n(n)-E_{n-1}(n-1)$ encodes primality and whether refinements might touch conjectures such as twin primes or other prime gap phenomena.
		
		\item \textbf{Computational complexity.}  
		Develop optimized algorithms for evaluating $E_m(n)$ in large parameter ranges, possibly exploiting recursive structure, sieve techniques, or fast divisor-sum methods.
		
		\item \textbf{Generalizations.}  
		Consider weighted versions $\sum_{k=1}^m w(k)\,(n\bmod k)$, or extensions to other algebraic structures (e.g. polynomial modular energies over finite fields).
	\end{enumerate}
	
	\medskip
	\noindent
	\textbf{Outlook.}  
	The modular energy encapsulates classical number-theoretic information in a novel aggregation of remainders.  
	Its algebraic exactness makes it accessible to combinatorial and computational methods, while its connections to divisor sums and congruence structures suggest deeper analytic significance.  
	Part~II will pursue these directions, focusing on asymptotic behavior, analytic continuations, and links with conjectural properties of primes and divisor functions.

\end{document}